\documentclass[11pt]{amsart}

\usepackage[english]{babel}
\usepackage{dsfont}
\usepackage{latexsym,amssymb, amsmath,  epsfig}
\usepackage{comment}

\usepackage{amsmath,amssymb,amsthm,color}
\usepackage{booktabs}

\newtheorem{coro}{Corollary}
\newtheorem{definition}{Definition}
\newtheorem{theorem}{Theorem}
\newtheorem{propo}{Proposition}
\newtheorem{remark}{Remark}

\newcommand{\prob}{\mathbb{P}}
\newcommand{\mean}{\mathbb{E}}
\newcommand{\real}{\mathbb{R}}
\newcommand{\corr}{\operatorname{corr}}
\newcommand{\ext}{\operatorname{ext}}
\newcommand{\conv}{\operatorname{conv}}
\newcommand{\cut}{\operatorname{CUT}}

\def\R{{\mathbb R}} 
 
\def\N{{\mathbb N}}

\newcommand{\Bn}{{\mathcal{B}_n}}

\newcommand\pix{\pi_{\bf x}}
\newcommand{\n}{$n$}
\newcommand{\cutn}{\ensuremath{\operatorname{CUT}(n)}}
\newcommand{\cutstarn}{\ensuremath{\operatorname{CUT}^\ast(n)}}

\newcommand{\unifdist}{\textsf{Unif}}
\newcommand{\berndist}{\textsf{Bern}}
\newcommand{\ind}{\mathds{1}}

\makeatletter
\newcommand\xleftrightarrow[2][]{\ext@arrow 0099{\longleftrightarrowfill@}{#1}{#2}}
\def\longleftrightarrowfill@{\arrowfill@\leftarrow\relbar\rightarrow}
\makeatother

\title{ Bernoulli Correlations \\ and\\ Cut Polytopes}
\author{Mark Huber}
\address{Claremont McKenna College}
\email{\tt mhuber@cmc.edu}

\author{Nevena Mari\'c}
\address{University of Missouri-St. Louis}
\email{{\tt maric@math.umsl.edu}}

\begin{document}
\maketitle

\begin{abstract}
\noindent Given $n$ symmetric Bernoulli variables, what can be said about their correlation matrix viewed as a vector? 
We show that the set of those vectors $R(\Bn)$ is a polytope and identify its vertices. Those extreme points correspond to correlation vectors associated to the discrete uniform distributions on  diagonals of the cube $[0,1]^n$. We also show that the polytope is affinely isomorphic to a well-known cut polytope ${\rm CUT}(n)$ which is defined as a convex hull of the cut vectors in a complete graph with vertex set $\{1,\ldots,n\}$. The isomorphism is obtained 
explicitly as $R(\Bn)= {\mathbf{1}}-2~{\rm CUT}(n)$.  As a corollary of this work, it is straightforward using linear programming to determine if a particular correlation matrix  is realizable or not.  Furthermore, a sampling method for multivariate symmetric Bernoullis with given correlation is obtained.  In some cases the method can also be used for general, not exclusively Bernoulli, marginals.
\vskip3mm
\noindent MSC 52B12, 60E05, 62H20
\end{abstract}

\section{Introduction}

Consider the question of admissible correlations among $n$ random variables $(X_1,\ldots,X_n)$ for which the marginal distributions are known.
This topic has a long history, dating partly back to the work of de Finetti~\cite{de1937} where the problem of maximum negative achievable correlation among $n$ random variables was studied. The general form of the problem was studied by Fr\'echet~\cite{frechet1951} and Hoeffding~\cite{hoeffding1940}  in a body of work which grew out of the problem originally  posed by  L\'evy~\cite{levy}.  

The big question is: can we completely describe set of correlation matrices for a given set of marginal distributions? For $n=2$ the answer is 
completely known in terms of Fr\'echet-Hoeffding bounds, so the interesting 
question is what happens in higher dimensions.  In this work we give
a partial answer to this question, showing that if a particular vector
calculated from the target correlations and marginals
falls into the ${\rm CUT}(n)$ polytope, then there does exist such a joint
distribution.  This condition is both necessary and sufficient in the 
case of symmetric Bernoulli marginals.

It is well known that a correlation matrix is symmetric positive semi-definite and has all diagonal elements equal to 1.  The set of all such matrices of order $n$ will be denoted by $\mathcal{E}_{n \times n}$. This convex compact set is sometimes called  the {\em elliptope}, a term coined by Laurent and Poljak \cite{laurent1995}. 

For Gaussian marginals, the entirety of $\mathcal{E}_{n \times n}$ can be 
realized, but this is the only nontrivial set of marginals for which the
question has been settled.
Surprisingly enough, for other common distributions very little is known. 
A reason is perhaps that much of correlation theory has been  constructed 
using normal marginals \cite{kotz2001}.  
In spite of wide usage of the correlation coefficient,  its bounds in a multivariate non-Gaussian setting has been mostly unexplored.

One case that has been partially explored is that of copulas.
A probability measure $\mu$ on $[0,1]^n$ is a {\em copula} if all its marginals are uniformly distributed on $[0,1]$. In a recent work of Devroye and Letac \cite{devroye2015copulas} it has been shown that every element in $\mathcal{E}_{n \times n}$ is a correlation matrix for some copula, for $n \leq 9$, but the authors believe that the statement does not hold for $n \geq 10$. 

In this paper we focus on symmetric Bernoulli variables for multiple reasons. 
First, it is the simplest distribution, with equally likely binary outcomes.
We say that a random variable $X$ has the Bernoulli distribution (write $X \sim \berndist(p)$) 
if $\prob(X = 1) = p$ and $\prob(X = 0) = 1 - p$.  
The symmetric Bernoulli distribution is the case when $p = 1/2$.

The second reason comes from Huber and Mari\'c~\cite{huberm2015} where this distribution was shown to be in a certain sense the most difficult problem:  for 
general marginals and some correlations it is possible to transform 
the problem into symmetric Bernoulli marginals.

To be precise, the problem is to simulate 
$(X_1,\ldots,X_n)$ where the correlation between each $X_i$ and $X_j$ are specified along with the marginal distribution of each $X_i$, or to determine
that no such random variables exist.

This problem, in different guises, appears in numerous fields: physics~\cite{Smith1981}, 
engineering~\cite{Lampard1968}, ecology~\cite{dos2008use}, and finance~\cite{lawrance1981new}, to name just a few. 
Due to its applicability in the generation of  synthetic optimization problems, it has also received special attention 
by the simulation community~\cite{Hill1994},~\cite{henderson2000}.

An excellent  overview of the developments in the field of generating multivariate probability densities with pre-specified margins can be found in  Dall'Aglio {\it et al.}~\cite{dall},  R\"uschendorf  {\it et al.}~\cite{Ruschendorf1996}, and Conway~\cite{Conway1979}. 



Consider a matrix $\Sigma$ that is potentially a correlation matrix for 
a particular choice of marginal distributions.  Then what the method 
of~\cite{huberm2015} does is build a second correlation matrix
$\Sigma_B$ such that if is possible to have a multivariate distribution
with symmetric Bernoulli marginals and correlation $\Sigma_B$, 
then it is possible to build a 
multivariate distribution with the original marginals and 
correlation matrix $\Sigma$.

The question of existence for general marginals 
then becomes the question of which matrices can be realized as correlation
matrices for $(B_1,\ldots,B_n)$ which are all marginally $\berndist(1/2)$.

It should be noted that  the answer for symmetric Bernoulli marginals will be a strict subset of $\mathcal{E}_{n \times n}$, even when $n$ is small.  
As a simple example consider 
\[
\begin{pmatrix} 1 & -0.4 & -0.4 \\ -0.4 & 1 & -0.4 \\ -0.4 & -0.4 & 1
  \end{pmatrix}.
\]
While this matrix is in the elliptope $\mathcal{E}_{3 \times 3}$, it cannot (see~\cite{huberm2015})
be the correlation matrix of three random variables with symmetric Bernoulli marginals.

Considering the applications, there is an obvious need to understand what are
the theoretically attainable correlations, so that inference using the correlations drawn from data can be done properly.
Chaganty and Joe~\cite{chagantyj2006} claim that there were errors in reports of efficiency calculations for
generalized estimating equations, such as Table 1 of Liang and
Zeger~\cite{liangz1986}, and with the analysis of real-life binary data using
the current generalized estimating equations software, caused by
the belief that any matrix in $\mathcal{E}_{n \times n}$ is a possible correlation matrix for a set of
binary random variables. In the same paper they were able to characterize 
the achievable correlation matrices when the marginals are Bernoulli.  When the dimension
is 3 their characterization is easily checkable (as for the 3 by 3 matrix
given above), in higher dimensions they give a number of inequalities that grows exponentially in the dimension. 
They also give an approximate method for checking attainability of the 
correlation matrix in higher dimensions.


 

In this paper we give a complete characterization of the correlation matrices for multivariate symmetric Bernoulli distributions by relating them to the well-known CUT polytope. This approach leads also to a novel sampling method from the desired marginals and correlations.

Let  $\Bn$ be a set of all $n$-variate symmetric Bernoulli distributions, $E_n= \{(i,j): 1 \leq i \neq  j \leq n \}$, and $R: \Bn \to [-1, 1]^{E_n}$ the correlation mapping. Here in place of a correlation matrix we focus on the {\em correlation vector} which contains elements above (or below) the diagonal of a correlation matrix placed in the same order, row by row. We show that $R(\Bn)$ is a polytope and identify its vertices. 

 Let $\pi_i$ for $i$ in $\{1,\ldots, 2^{n-1}\}$ be the uniform distribution over the end points of the $i$-th diagonal of the $n$-dimensional cube $[0,1]^n$. Then our main results is as follows.
 
 \begin{theorem}\label{THM:corr}
$R(\Bn)$ is a polytope with vertices $\{R(\pi_1),...,R(\pi_{2^{n-1}})\}$. That is, a vector $\rho \in [-1, 1]^{E_n}$ is a correlation vector for some distribution $\mu \in \Bn$ if and only if it can be written as a convex combination of $\{R(\pi_1),...,R(\pi_{2^{n-1}})\}$.
\end{theorem}
We also uncover a striking relation between this polytope and the {\em cut polytope} CUT(\n), which is defined as a convex hull of the cut vectors in a complete graph with vertex set $\{1,\ldots,n\}$. The cut polytopes play an important role in combinatorial optimization, as they can be used to formulate the max-cut problem, which has many applications in various fields \cite{dezalaurent1997}, \cite{ziegler2000}. Vertices of the CUT are all 0-1 vectors, meaning that each coordinate is either 0 or 1. A relation between the polytopes is given the following theorem.


%

\begin{theorem}\label{THM:polyrel} For $\rho \in [-1, 1]^{E_n}$:
$\rho \in R(\Bn)$ if and only if  ${\mathbf{1}} - 2 \rho  \in \mbox{CUT}(n)$.
 \end{theorem}
 
 The remainder of the text is organized in the following way. In Section \ref{sec:bernoulli} we introduce the problem of Bernoulli correlations via their agreement probabilities, starting off from a prior work \cite{huberm2015} and study distributions $\pi_i$ as (some) extreme points of $\Bn$.  Cut polytopes are introduced in Section \ref{sec:cut}, where also an important theorem of Avis \cite{avis1977} is analyzed. The theorem provides a certain probabilistic context of cut polytopes, and we are able to extend it in connection to $\Bn$. In Section \ref{sec:berncut} we derive our main results. The asymmetric Bernoulli case is also discussed here. Section \ref{sec:applications} is focused on applications and how our findings can be used in practice via linear programming. A sampling method for $n$-variate symmetric Bernoullis with given correlation is outlined as well. This section also contains a worked through example using difficult marginals. Finally in Section \ref{sec:discussion} we discuss our results in a larger context.
 
\section{Bernoulli correlations and agreement probabilities}\label{sec:bernoulli}

\begin{definition}\label{def:nbern}
An $n$-variate symmetric Bernoulli distribution $\mu$ is a probability measure on $\{0,1\}^n$ such that all the marginals are \berndist(1/2), that is
\[
\sum_{{\bf x} \in \{0,1\}^n:~{\bf x}(k)=0} \mu({\bf x})=  \frac{1}{2} \hskip10mm  \mbox{for } k=1,\ldots,n.
\]
Let $\Bn$ be a set of all such measures.
\end{definition}

\begin{remark}
Clearly $\Bn$ is a convex set: if $\mu_1, \mu_2 \in \Bn$ then, for any $ \alpha \in (0,1)$,  $\alpha \mu_1 + (1-\alpha) \mu_2 \in \Bn$.  Since it is 
described by a finite set of linear equalities, $\Bn$ is closed subset of $\R^{2^n}$ and is bounded. Moreover $\Bn$ is a polytope \cite{devroye2015copulas}.
\end{remark}


\begin{definition}
If $(B_1,\ldots,B_n)\sim \mu,~ \mu \in \Bn$ then for all $i,j \in \{1,\ldots,n\}$, the probability that $B_i$ and $B_j$
have the same value is called the { agreement probability}, and denoted as $\lambda^{\mu}(i,j)=\prob_{\mu}(B_i=B_j)$. \\
The map that takes $\mu$ to $\lambda^\mu$ is denoted by $\Lambda$.  
\end{definition}

Now consider the correlation between two $\berndist(1/2)$ random variables:
 \[
 \rho^{\mu}(i,j)=\corr(B_i,B_j)={4 \mean_{\mu} B_iB_j-1}.
 \]
\begin{remark}
Correlation is related to the agreement probability between the variables in a linear way:
\begin{align*}
 \lambda^{\mu}(i,j)=\frac{1}{2}(1+ \rho^{\mu}(i,j)).
\end{align*}
\end{remark}

In \cite{huberm2015} Huber and Mari\'c studied elements of $\Bn$ via their agreement probabilities and they were able to provide necessary and sufficient conditions for an agreement matrix to be attainable up to the dimension 4.
It can be verified that those conditions, given in Theorem 3 of \cite{huberm2015},  place $\lambda$ in the following polytopes:
 \begin{itemize}
\item $n=2$: Interval $[0,1]$.
\item $n=3$: Tetrahedron with vertices $(1,1,1), (1,0,0), (0,1,0), (0,0,1)$.
\item $n=4$: $(\lambda_{12}, \lambda_{13}, \lambda_{14}, \lambda_{23}, \lambda_{24}, \lambda_{34})$ belongs to a 6-polytope with 8 vertices:$(1,1, 1, 1, 1, 1 )$, $( 0, 1, 1, 0, 0, 1)$, $(1, 1, 0, 1, 0, 0 )$, $(1, 0, 0, 0, 0, 1)$, $(0, 1, 0, 0, 1, 0)$, $(0, 0, 1, 1, 0, 0)$, $( 0, 0, 0, 1, 1, 1)$, $(1, 0, 1, 0, 1, 0)$. 
Here each appropriate 3-dimensional projection (that corresponds to agreement probabilities in any given subset of 3 variables; for example along first, second, and fourth coordinate) bears a tetrahedron from the case $n=3$, as it should be.
\end{itemize}

These results stem from~\cite{huberm2015} but also follow directly
from Theorem~\ref{THM:corr}.
 
 \subsection{Diagonals}
 
Let $ {\bf x}=(x_1,\ldots,x_n) \in \{0,1\}^n$.  We can 
think of ${\bf x}$ as a vertex of the $n$-dimensional cube $[0,1]^n$. 
Let ${\bf 1} = (1,\ldots,1)$. The diagonal associated to the vertex ${\bf x}$ (and its ``opposite" ${\bf 1-x}$) is set  
\[
D_{\bf x}:=\{(x_1,\ldots,x_n), (1-x_1,\ldots,1-x_n)\}. 
\]

 \begin{definition} 
  For every ${\bf x} \in \{0,1\}^n$, the (discrete) uniform distribution over the 
diagonal $D_{\bf x}$ is denoted by $\pix$.  That is $\pix({\bf x})=\pix({\bf 1 -x})=1/2$. 
\end{definition}

\begin{remark}
Note that $\pix \in \Bn$ and also $\pix = \pi_{\bf 1-x}$.
\end{remark}

  The number of vertices ${\bf x}$ is  finite, so they can be ordered. Every vertex of the $n$-dimensional cube $[0,1]^n$ can be mapped to a positive integer via $b: \{0,1\}^n \to \N$:
  \[
b(x_1,x_2,\ldots,x_n) = 1 + \sum_{j=1}^n x_j 2^{n-j}.  
\]
  
  Each set $D_{\bf x}$ contains one vector ${\bf y}$ 
  whose first coordinate is 0, and 
  for this vector $b({\bf y}) \in \{1,\ldots,2^{n-1}\}$.  
  To simplify the notation, for $k \in \{1,\ldots,2^{n-1}\}$ let
  \vskip2mm
  \begin{center}
  $\pi_k$ :=  the uniform distribution over $D_{\bf x}$ such that $\exists {\bf y} \in D_{\bf x}$ with $b({\bf y}) = k$ .
\end{center}
\vskip2mm
It should be noted that $\{\pix\}_{{\bf x} \in \{0,1\}^n}$ and $\{\pi_k\}_{k=1}^{2^{n-1}}$ refer essentially to the same set of distributions, with a difference being that in the ordered notation there are no equal elements, e.g. $\pi_1=\pi_{\bf 0}=\pi_{\bf 1}$. However, both notations are useful, and we will be using them concurrently throughout the text. 
\vskip3mm
 Before we state the first result highlighting the importance of distributions ${\mathbf{\pi}}$, a few additional definitions are needed. As in Gr\"unbaum~\cite{branko2003}, $v \in P$ is an extreme point 
of the polytope $P$ if 
$(\forall a,b\in P)(\exists \beta \in (0,1))
 (\beta a + (1 - \beta b) = v \Rightarrow a = b = v)$.  Let
$\ext(P)$ denote the extreme points of the polytope $P$.

 \begin{propo} \label{propo:piext}
 For every $k=1,\ldots,2^{n -1}$, a measure $\pi_k$ is an extreme point of $\Bn$. That is $ \{\pi_1,\ldots, \pi_{2^{n-1}}\} \subseteq \ext(\Bn)$.
 \end{propo}
 
 \proof
 
Let ${\bf x},{\bf y} \in \{0,1\}^n$ and $\pix$ be a convex linear combination of $\mu$ and $\nu$ in $\Bn$. Then for ${\bf y} \neq {\bf x}, {\bf 1-x}$, $\pix({\bf y})=0$ and therefore $\mu({\bf y})=\nu({\bf y})=0$. Suppose $\mu({\bf x})=p$ and $\mu({\bf{1-x}})=1-p$, for $0<p<1$. As $\mu$ has \berndist(1/2) marginals, it follows that $p=1/2$ and therefore $\mu=\pix$. Analogously, $\nu$ is also equal to $\pix$ and hence, $\pix$ is an extreme point of $\Bn$.

 \endproof

\section{CUT polytopes}\label{sec:cut}

Let $G = (V,E)$ be a graph with vertex set $V$ and edge set $E$. For $S \subseteq V$
 a {\em cut} of the graph is a partition $(S,S^C)$ of the vertices.
The {\em cut-set} consists of all edges that connect a node in 
$S$ to a node not in $S$. 

Let $V_n=[n]= \{1,\ldots,n\}$, $E_n= \{(i,j); 1 \leq i \neq  j \leq n \}$,  and $K_n=(V_n, E_n)$ be a complete graph with the vertex set $[n]$.
\begin{definition}\label{def:cut}
For every $ S \subseteq [n]$ a vector $\delta(S) \in \{0,1\}^{E_n}$, defined as 

\begin{align*}
\delta(S)_{ij}= \left\{ \begin{array}{cc}
1, & \mbox{ if }  |S \cap \{i,j\}|=1 \\
0, & \mbox{ otherwise }
  \end{array} \right.
  \end{align*}
  for $(1 \leq i < j \leq n)$, is called a {\em cut vector} of $K_n$.
  
\end{definition}

The {\em cut polytope} CUT(\n) is the convex hull of all cut vectors of $K_n$:
\[
\text{CUT}(n) = \conv\{ \delta(S): S \subseteq [n] \}.
\]

\begin{remark}
Since every cut vector is a vertex of \cutn, there are $2^{n-1}$ vertices 
of this polytope~\cite{ziegler2000}.
\end{remark}

Each $\delta(\cdot)$ is a $0/1$-vector (every coordinate value is either 0 or 1). The convex hulls of finite sets of $0/1$-vectors are called {\em 0/1-polytopes}, out of which cut polytopes are a sub-class.
An excellent lecture on 0/1-polytopes, including CUT,  is given in Ziegler~\cite{ziegler2000}. More thorough treatment of cut polytopes can be found in Deza and Laurent~\cite{dezalaurent1997}. 
The cut polytopes play an important role in combinatorial optimization, as they can be used to formulate the max-cut problem, which has many applications in various fields, like statistical physics, in relation to spin glasses~\cite{dezalaurent1997}.

All the symmetries of the cube $[0,1]^n$ transform 0/1-polytopes into 0/1-polytopes \cite{ziegler2000}. In particular, a symmetry is obtained by replacing some coordinates $x_i$ by $1-x_i$, which is called {\em switching}. Two 0/1-polytopes {\em P} and {\em P'} are {\em 0/1- equivalent} if a sequence of  switching and coordinate permuting operations can transform {\em P} into {\em P'}. We define now a polytope that is  0/1- equivalent to \cutn ~ and also useful for our further analysis.

\begin{definition}
  CUT$^*(n)$ is a polytope obtained by applying 0/1 switching operation to {\em all vertices} of \cutn. 
\end{definition}

A relation of cut polytopes with probability spaces is given in the following theorem. We cite here a version of the theorem given in the book by Deza and Laurent \cite{dezalaurent1997}, but the authorship dates back to Avis \cite{avis1977}.
\begin{theorem}[Avis \cite{avis1977}] \label{thm:avis}
For a vector $d \in \R^{E_n}$  the following statements are equivalent
\begin{align*}
1)\, &d \in \cutn \\
2) \, &\text{There exists a probability space } (\Omega, \mathcal{A}, \nu) \text{ and events }  A_1,\ldots, A_n \in \mathcal{A}\\
  & \text{such that } d_{ij}= \nu(A_i \triangle A_j)  \text{ for all }  1 \leq i <j \leq n.
 \end{align*}
\end{theorem}
Here $A \triangle B = (A \setminus B) \cup (B \setminus A)$ denotes the symmetric
difference between sets $A$ and $B$.

This theorem is useful for us for the following reason: to every event in a probability space one can assign a Bernoulli random variable, as its indicator. In order to obtain a {\em symmetric} \berndist, the probability of the event has to be equal to 1/2. As we are going to show, that is exactly what happens in the setting of the above theorem.

\subsection{On Theorem~\ref{thm:avis} }

We look into the proof of Theorem \ref{thm:avis}, given in \cite{dezalaurent1997} Proposition 4.2.1., of the part $1) \Rightarrow 2)$ and will show that events $A_i$, as defined in that proof, have to be of probability 1/2.

 Let $d  \in \cutn$, then  $d = \sum_{S \subseteq [n]} \beta_S \delta(S)$ for some $\beta_S \geq 0$ and $\sum \beta_S =1$. Note that 
$\delta(S)=\delta(S^c)$ and that this vector ultimately participates only once in the convex representation  of $d$, so, without loss of generality, one can assume that $\beta_S= \beta_{S^c}$. Now, the probability space $(\Omega, \mathcal{A}, \nu)$ is defined as follows.  Let $\Omega =\{ S: S \subseteq [n]\}$, $\mathcal{A}$ the family of subsets of $\Omega$ and $\nu$ a probability measure, for $A \in \mathcal{A}$ defined by
\[ 
\nu(A)= \sum_{S \in A} \beta_S.
\]
Setting $A_i=\{ S \in \Omega: i \in S\}$ one obtains events with the desired property, namely for which $d_{ij}=\nu(A_i \triangle A_j) $.
In order to calculate $\nu(A_i)$ we observe that 
\begin{align}
\nu(A_i)= \sum_{S \in A_i} \beta_S=  \sum_{S \subseteq [n]} \beta_S \ind(i \in S)
\end{align}
and for its complement $A_i^c=\{ S \in \Omega: i \notin S\}$
\begin{align*}
 \nu(A_i^c)&= \sum_{S \in A_i^c} \beta_S = \sum_{S \subseteq [n]} \beta_S \ind(i \notin S)=\sum_{S \subseteq [n]} \beta_{S} \ind(i \in S^c)\\
 &= \sum_{S \subseteq [n]} \beta_{S^c} \ind(i \in S^c) = \nu(A_i).
\end{align*}
Therefore $\nu(A_i)= 1/2$ for $i=1,\ldots,n$ and we can formulate the following corollary.

\begin{coro}\label{coro:avis2}
For a vector $d \in \R^{E_n}$  the following statements are equivalent
\begin{align*}
1)\, &d \in \cutn \\
2) \, &\text{There exists a probability space } (\Omega, \mathcal{A}, \nu) \text{ and events }  A_1,\ldots, A_n \in \mathcal{A}\\
  & \text{such that } \nu(A_i) = \frac{1}{2}  \text{ for all }  1 \leq i \leq n  \text{ and } d_{ij}= \nu(A_i \triangle A_j)  \text{ for all }  1 \leq i <j \leq n.
 \end{align*}
\end{coro}

\begin{remark}
In the $2) \Rightarrow 1)$ part of the proof of the Theorem \ref{thm:avis}, the probabilities of the events $A_i$ do not play any role, so that  direction of the statement remains true in the above Corollary.
\end{remark}

\section{Bernoulli agreements and CUT*} \label{sec:berncut}
To every event $A$ in a probability space can be associated its {\em indicator} $\ind(A)$,  a random variable that takes binary values: $1$ if the event occurs and $0$ otherwise. The indicator has Bernoulli distribution with parameter that equals the probability of the event $A$.
\vskip2mm
Suppose now that the condition 2) from Corollary \ref{coro:avis2} is satisfied and define $D_i = \ind(A_i)$ for $i =1,\ldots,n$. Clearly $D_i \sim \berndist(1/2)$ and for all $i,j \in E_n$
\begin{align} \label{eq:d}
d_{ij}= \nu(A_i \triangle A_j)=\prob_{\nu}(D_i \neq D_j)=1- \lambda^{\nu}(i,j).
\end{align}
\vskip2mm
From (\ref{eq:d}) and Corollary \ref{coro:avis2} then it follows
\vskip1mm
\begin{align} \label{eq:lambdaincut}
\lambda^{\nu}(i,j) \in CUT^*(n).
\end{align}
\vskip3mm

\begin{propo} \label{propo:cutunif}
$CUT^*(n)$ is a polytope with vertex set $\{\lambda^{\pi_{\bf x}}:  {\bf x} \in \{0,1\}^n\}$.

\end{propo}

\begin{proof}
  Let ${\bf x}=(x_1,\ldots,x_n) \in \{0,1\}^n$.  Fix $i,j \in [n]$ and consider
  $\lambda^{\pi_{\bf x}}(i,j)$
  \[
  \lambda^{\pi_{\bf x}}(i,j) = \sum_{{\bf y} \in \{0,1\}^n : y_i=y_j} {\pi_{\bf x}}({\bf y}).
  \]
  If $x_i = x_j $, then also $1 - x_i = 1 - x_j$, and the above sum equals to ${\pi_{\bf x}}({\bf x})+{\pi_{\bf x}}({\bf 1-x})= 1$.

  Similarly, if $x_i \neq x_j$, then $1 - x_i \neq 1 - x_j$ and every summand above equals 0.  Hence
  \begin{equation} \label{eq:lambdaind}
  \lambda^{\pi_{\bf x}}(i,j) = \ind(x_i = x_j).
  \end{equation}

  To establish a relation between the agreement vector and a cut vector
  introduce $S_{\bf x}=\{ i \in [n] : x_i=1\}$,  the set of all
  coordinates whose value in ${\bf x}$ is 1.  If exactly one of $i, j \in [n]$ is in $S_{\bf x}$ then $x_i \neq x_j$, otherwise $x_i = x_j$. Therefore, from (\ref{eq:lambdaind}) it follows 
 \[
  \lambda^{\pi_{\bf x}}(i,j)= 1- \ind (x_i \neq x_j)=1- \ind( |S_{\bf x} \cap \{i,j\}|=1) \hskip7mm \forall i,j \in [n]. 
  \]
Applying Definition \ref{def:cut} we get the desired relationship
 \begin{align} \label{eq:lambdadelta}
 \lambda^{\pi_{\bf x}}= {\bf 1}- \delta(S_{\bf x}).
 \end{align}
 Note also that a map from $ \{0,1\}^n $ to $ \mathcal{P}([n])$ (the power set of $[n]$) where $ {\bf x} \mapsto S_{\bf x}$ is a bijection and that every $S \subseteq [n]$ can be identified as $S_{\bf x}$ for some ${\bf x} \in \{0,1\}^n$.
 From the equation (\ref{eq:lambdadelta}) and the definition of cut polytope then we have
  \begin{align*}
  \cutn &= \conv\{\delta(S): S \subseteq [n]\}\\
  &=  \conv\{\delta(S_{\bf x}):  {\bf x} \in \{0,1\}^n\}= 
       \conv\{ {\bf 1}- \lambda^{\pi_{\bf x}}:  {\bf x} \in \{0,1\}^n\}.
  \end{align*}
  Finally, since $\cutstarn$ is a 0/1 switched image of  $\cutn$, it follows that $\cutstarn=\conv\{\lambda^{\pi_{\bf x}}:  {\bf x} \in \{0,1\}^n\}$.
  \vskip3mm
  As every cut vector is a vertex of $\cutn$, then  it follows that 
  $\{\lambda^{\pi_{\bf x}}:{\bf x} \in \{0,1\}^n\}$ must be the vertex set
  of the polytope \cutstarn.
\end{proof}

The next theorem says the set of all agreement vectors (for 
$n$-dim symmetric Bernoulli distributions)  is exactly the known polytope 
$\cutstarn$. Proposition \ref{propo:cutunif} gives a probabilistic interpretation of its extreme points,
and consequently, we are able to completely describe every $\lambda^\mu$.

 \begin{theorem} \label{thm:lambdacut}
  $ \Lambda(\Bn)=\cutstarn.$ 
\end{theorem}

\begin{proof}
  ($ \Rightarrow$:) Let $\mu$ be a measure in $\Bn$. Then, by (\ref{eq:lambdaincut}), $ \lambda^\mu \in \cutstarn$  and by Proposition \ref{propo:cutunif} follows that 
  $ \lambda^\mu \in \conv\{\lambda^{\pi_{\bf x}}:  {\bf x} \in \{0,1\}^n\}$. Hence 
  \[
  \Lambda(\Bn) \subseteq \cutstarn.
  \]
  \vskip5mm
  ($ \Leftarrow$:) Conversely, suppose that $\lambda \in \cutstarn$ 
  and then we want to show that there is $\mu \in \Bn$ such that 
  $\lambda^\mu = \lambda$.  Recall that for each diagonal distribution
  $\pi_{\bf x}$ there is a unique label $k \in \{1,\ldots,2^{n-1}\}$ such that 
  $\pi_{\bf x} = \pi_k$.  Then write $\lambda$ as a convex linear combination
  of the $\lambda^{\pi_k}$.
 \begin{align*}
   \lambda = \sum_{k=1}^{2^{n-1}} \alpha_k \lambda^{\pi_k}, \quad
     \sum_{k=1}^{2^{n-1}} \alpha_k = 1
 \end{align*} 
  Then the distribution is 
  \[ 
  \mu = \sum_{k=1}^{2^{n-1}} \alpha_k \pi_k.
  \] 
  Being a convex combination of  $(\pi_1,\ldots, \pi_{2^{n-1}}) \in \Bn$, $\mu$ is also in  $\Bn$.  Fix $i,j \in \{1,\ldots,n\}$ and 
consider the probability that for $(B_1,\ldots,B_n) \sim \mu$, that $B_i = B_j$.  Since $\mu$ is a convex mixture of other measures,
\[
\lambda^{\mu}(i,j) = \sum_{k=1}^{2^{n-1}} \alpha_k \lambda^{\pi_k}(i,j) 
  = \lambda(i,j).
\]
\end{proof}
   
  The original motivation was to understand the attainable correlation
  matrices.  Recall the correlation mapping $R$: $ \mu \mapsto \rho^\mu$,
  and  recall the relation between agreement probability and the 
  correlation, which is obviously a linear bijection: 
  \[
  R({\mu})= 2\lambda^{\mu}- \mathbf{1}.
  \]
  Our original goal was to describe the set $R(\Bn) = 2\Lambda(\Bn) - {\mathbf{1}}$.
  Since $\Lambda(\Bn)$ has extreme points corresponding to $\lambda^{\pi_i}$,
  so does $R(\Bn)$, proving Theorem~\ref{THM:corr}.
  
  \begin{remark}
  From Proposition \ref{propo:piext} we know that for every ${\bf x } \in \{0,1\}^n$, $\pix$ is an extreme point of $\Bn$. Since $R$ is an affine transformation we also know that $\text{ext}(R(\Bn)) \subseteq R(\text{ext}(\Bn))$. Due to Theorem \ref{thm:lambdacut}, even we do not have a complete description of $\Bn$ (i.e we do not know all its extreme points) we are still able to completely describe $R(\Bn)$.
  \end{remark}
  Note also that for $\delta \in [0,1]^{E_n}$,  $\delta \in \text{CUT}(n) \Leftrightarrow {\mathbf{1}}- \delta \in CUT^*(n)$. Now Theorem \ref{THM:polyrel} follows directly from Theorem \ref{thm:lambdacut}.

  \subsection{Relation to the correlation polytope}  
  The polytope $R(\Bn)$ does describe
  the set of allowable correlation matrices, but it is not directly 
  a {\em correlation
  polytope}, which is typically defined (see~\cite{fiorinimptdw2012}) as
  \[
  \operatorname{COR}(n) = \text{conv}\{bb^T \in \real^{n\times n}|b \in \{0,1\}^n\}.
  \]
   However, since there is an affine isomorphism between CUT($n$) and COR($n-1$) (see \cite{dezalaurent1997}) they can be related in a similar way as we did here. 
   
   \subsection{Asymmetric $n$-variate Bernoulli}
It should be noted that discovered relation between CUT($n$) and $\Bn$ does not extend to asymmetric multivariate Bernoulli distributions. 
It is enough to analyze the bivariate case with equal marginals. 

\par The correlation between two \berndist($p$) random variables  belongs to the interval $[\rho_{\min},1]$. Maximum correlation in case of equal marginals, always equals to 1 and the minimum correlation $\rho_{\min}$ can be calculated using Fr\'echet-Hoeffding bounds~\cite{frechet1951,hoeffding1940} 
\begin{align*}
\rho_{\min{}}=  \left \{ \begin{array}{ll}
 -(1-p)/p, & \mbox{for } p \geq 1/2 \\
  -p/(1-p), & \mbox{for } p \leq 1/2.
\end{array} \right.
\end{align*}
It is clear now that only for $p= 1/2$, $\rho_{\min}=-1$ and possible correlations equal to the entire interval $[-1,1]$, while for any other value of $p$ it is a strict subinterval of $[-1,1]$. From the linear relationship between the agreement probability $\lambda$ and $\rho$, it follows that, again only for $p=1/2$ it is true $0\leq \lambda \leq 1$, while for other values of $p$ these bounds are not sharp. For example, for $p=3/4$, $-1/3 \leq \rho \leq 1$, and $ 1/3 \leq \lambda \leq 1$.\\
In two dimensional case the cut polytope is known to be $\cut(2) =[0,1]= \cut^*(2)$ so it corresponds to $\Lambda(\Bn)$ only in the symmetric case.

\section{Applications}\label{sec:applications}

\subsection{Determining feasibility}

Once the vertices of the polytope $R(\Bn)$ have been determined, it is 
straightforward using linear programming to determine if a particular 
correlation vector $\rho: E_n \to [-1,1]$ is realizable.  

For $k \in 1,\ldots,2^{n-1}$, let $x_k(i)$ be the $i$th bit in the 
binary representation of the number $k - 1$.  Then let
\[
R(\pi_k)(i,j) = 2\ind(x_k(i) = x_k(j)) - 1.
\]
The goal is to find nonnegative $\alpha_1,\ldots,\alpha_{2^{n-1}}$ such that, for all $i < j$
\[
\alpha_1 R(\pi_1)(i,j) + \cdots + \alpha_{2^{n-1}}R(\pi_{2^{n-1}})(i,j) = \rho(i,j)
\]
and $\alpha_1 + \cdots + \alpha_{2^{n-1}} = 1$.

Let $\rho^{\text{aug}} =[ \rho ~~1]$. These $\binom{n}{2} + 1$ 
equations can be written as the system $M \alpha = \rho^{\text{aug}}$ 
subject to $\alpha \geq 0$ for
appropriate choice of $M$ and $\rho^{\text{aug}}$.  It can be determined
if this polytope is nonempty by solving the {\em Phase I} linear program
\begin{align*}
\min z_1 + \cdots + z_{2^{n-1}} &   \\
  M\alpha + z &= \rho^{\text{aug}} \\
  \alpha, z \geq 0.
\end{align*}

This linear program has a feasible solution of $\alpha = 0$ and 
$z = \rho^{\text{aug}}$.  
Then the original polytope is nonempty if and only if this linear program
has a solution with objective function value $0$.

The free lpsolve package within R was used to test how possible this
would be.  On a problem with $n = 15$ it took roughly eight seconds 
to solve on an Intel i5@3.30 Ghz.  Of course the number of variables is
growing exponentially with $n$, so this method will grow rapidly as
$n$ grows.

\subsubsection{\bf Generating multivariate \berndist(1/2) with given correlation vector $\rho$}

The vector $(\alpha_1,\ldots,\alpha_{2^{n-1}})$ obtained above, using linear programming, can be used to sample $(B_1,\ldots, B_n)$ with desired correlation in the following way. Let us label the vertices of $n$-dimensional cube $[0,1]^n$ whose first coordinate is 0:
\[
{\bf v}_k := \{ {\bf y} \in \{0,1\}^n : \pi_k({\bf y})=1/2 ~\text{and}~ {\bf y}(1)=0 \}. 
\]
Then for $k=1,\ldots, {2^{n-1}}$
\[
 (B_1,\ldots, B_n)= {\bf v}_k ~\text{or}~\mathbf{1}-{\bf v}_k ~\text{with probability}~ \alpha_k/2.
\]

\subsection{Example}  To illustrate how the above algorithm can be used for general margins, 
consider the following problem with $n = 4$.  Suppose 
$X_1$ is uniform over $[0,1]$, $X_2$ is exponential with mean 2, 
$X_3$ equals 1 with probability 0.3 and 4 with probability 0.7, and 
$X_4$ is a standard normal random variable.  The goal is to generate
$(X_1,X_2,X_3,X_4)$ from a 
multivariate distribution such that they have correlation matrix
\[
\Sigma = \begin{pmatrix}
  1 & 0.2 & -0.1 & -0.4 \\
  0.3 & 1 & -0.4 & 0.3 \\
  -0.2 & 0 & 1 & -0.2 \\
  -0.4 & 0.3 & -0.2 & 1 
\end{pmatrix}
\]

The first thing to do is to check if this correlation matrix is even 
possible with these marginals.  Recall that if a random variable
$X$ has cdf $F_X$, and $F^{-1}(u) =  \inf\{a:F_X(a) \geq u\}$, then 
for $U \sim \unifdist([0,1])$, we have that $F_X^{-1}(U) \sim X$.  
This is called the inverse transform method of converting uniforms into
variables with the same distribution as $X$.

Therefore, if a random vector $(U_1,U_2,U_3,U_4)$ is drawn where
the $U_i$ are marginally uniform, then 
$(F_{X_1}^{-1}(U_1),F_{X_2}^{-1}(U_2),F_{X_3}^{-1}(U_3),F_{X_4}^{-1}(U_4))$ 
has the correct marginals for the $X_i$.

Note that if $U$ is uniform over $[0,1]$, then so is $1 - U$.  The
variables $U$ and $1 - U$ are said to be {\em antithetic} variates.
With that in mind, one way to draw $(U_1,U_2,U_3,U_4)$ is to 
set $U_i = U$ or $U_i = 1 - U$ for each $i$.

This is where a symmetric Bernoulli random vector $(B_1,B_2,B_3,B_4)$
enters the picture.  For each $i$, if 
$B_i = 1$ then $U_i = U$ and $X_i = F_{X_i}^{-1}(U)$, and if $B_i = 0$ then 
$U_i = 1 - U$ and $X_i = F_{X_i}^{-1}(1-U)$.  In other words,
\[
X_i = B_i F_{X_i}^{-1}(U) + (1 - B_i) F_{X_i}^{-1}(1 - U).
\]

In this way, the problem of generating
$(X_1,X_2,X_3,X_4)$ with the correct marginals and correlations
can be transformed into
a problem of generating $(B_1,B_2,B_3,B_4)$ with symmetric Bernoulli marginals
and perhaps different correlations. This idea was explored thoroughly in \cite{huberm2015}.

The Fr\'echet-Hoeffding bound~\cite{frechet1951,hoeffding1940} 
then says that the correlation between $X_i$ and $X_j$ is maximized
when this method is used with $B_i = B_j$ and minimized with $B_i = 1 - B_j$.
For $X_2$, $F_{X_2}^{-1}(u) = -2\ln(1-u)$.  For $X_3$, 
$F_{X_3}^{-1}(u) = \ind(u \leq 0.3) + 4 \cdot \ind(u > 0.3)$.  A straightforward
calculation then 
gives the maximum and minimum possible correlation between $X_2$ and $X_3$ (see e.g. Whitt~ \cite{whitt1976}). 
\begin{align*}
\corr(F_{X_2}^{-1}(u),F_{X_3}^{-1}(u)) &= 0.544828\ldots \\
\corr(F_{X_2}^{-1}(u),F_{X_3}^{-1}(1-u)) &= -0.78818\ldots.
\end{align*}
Since the target correlation of $-0.4$ lies within
$[-0.78818,0.544828]$, it is possible to achieve.  

Because $-0.4$ falls in the interval, it 
can be written as a convex linear combination
of $-0.78818$ and $0.544828$:
\[
-0.4 = -0.78818(1-0.291209)+.544828(0.291209).
\]

This means that we want a 0.291209 chance that $U_2 = U_3$, and 
a $1 - 0.291209$ chance that $U_2 \neq U_3$, or equivalently, a  
0.291209 chance that $B_2 = B_3$ and 
a $1 - 0.291209$ chance that $B_2 \neq B_3$.

Let $B_2$ and $B_3$ be symmetric Bernoulli random
variables with a 0.291209 probability of being equal.  Then they have a 
correlation of $2(0.291209) - 1 = -0.417582.$  Performing this
calculation for all the different pairs of random variables gives
a correlation matrix for the symmetric Bernoulli random variables $(B_1,B_2, B_3, B_4)$ of
\[
\Sigma_B = \begin{pmatrix}
  1 & 0.230940 & -0.125988 & -0.409330 \\
  0.230940 & 1 & -0.417582 & 0.332154  \\
  -0.125988 & -0.417582 & 1 & -0.263598 \\
  -0.409330 & 0.332154 & -0.263598 & 1
  \end{pmatrix}.
\]

This correlation matrix (i.e. correlation vector $\rho$ = (0.230940, -0.125988, -0.409330, -0.417582, 0.332154, -0.263598)) corresponds to a linear program which was
solved using lpsolve in R, and the solution was found to be 
$(\alpha_1,\ldots,\alpha_8)$, which to four significant figures is 
\[
 \begin{pmatrix} 0.04337 & 0.1284 & 0.2450 & 0.1985 & 
  0.006894 & 0.2582 & 0 & 0.1193 \end{pmatrix}.
\]

Consider the sixth component $0.2582$.  
Since $6 - 1 = 5$ has binary expansion
$0101$, this means that there is a $0.2582$ chance that $U_2 = U_4 = U$ 
and $U_1 = U_3 = 1-U$.
Note that both $(B_1,B_2,B_3,B_4) = (0,1,0,1)$ and
$(B_1,B_2,B_3,B_4) = (1,0,1,0)$ give rise to $U_1 = U_3$ and $U_2 = U_4$.  
So in order to make a symmetric 
Bernoulli distribution, there is equal probability of either of 
those two vectors occurring.  That is,
$\prob((B_1,B_2,B_3,B_4) = (0,1,0,1)) = 
 \prob((B_1,B_2,B_3,B_4) = (1,0,1,0)) = 0.2582/2.$

This algorithm  for drawing a multivariate distribution for 
this example can then be given as follows.
\begin{enumerate}
  \item Draw $U$ uniformly over $[0,1]$.  Draw $Y$ so that 
  $\prob(Y = j) = \alpha_j,~j=1,\ldots,8$.
  \item Write $Y-1$ in binary notation.  That is, 
       find $b_i \in \{0,1\}$ so that $Y - 1 = \sum_{i=1}^{4} 2^{i-1} b_i$.
  \item For each $i \in \{1,\ldots,4\}$ if $b_i = 1$ then let $U_i = U$,
        otherwise let $U_i = 1 - U$.
  \item Let $X_1 = U_1$, $X_2 = -2\ln(1-U_2)$, $X_3 = \ind(U_3 \leq 0.3) + 4
        \cdot \ind(U_3 > 0.3)$,
        $X_4 = \Phi^{-1}(U_4)$.
\end{enumerate}

Here $\Phi$ is used in the usual fashion as the cdf of a standard normal.

\section{Discussion}\label{sec:discussion}
 The set of $n \times n$ correlation matrices, the elliptope $\mathcal{E}_{n \times n}$ is a nonpolyhedral convex set with a nonsmooth boundary. The extreme points of the elliptope have not been explicitly determined, but there exist characterization results on the rank one and two extreme points, done by Ycart \cite{ycart1985} (see also Li and Tam \cite{li1994note} and Parthasarathy \cite{parthasarathy2002}).
 Laurent and Poljak \cite{laurent1995}  proved that cut matrices (analogous to cut vectors) are actually vertices (that is, extreme points of rank one) of the elliptope  and that $\mathcal{E}_{n \times n}$ can be seen as a nonpolyhedral relaxation of the cut polytope. In view of theorems proved here it is clear that the vertices of $\mathcal{E}_{n \times n}$ correspond precisely to symmetric Bernoulli correlations. 

\subsection*{Acknowledgments} 
We are grateful to G\'erard Letac for sharing his ideas with us and for inspiring discussions. 

\bibliographystyle{plain}
\bibliography{Berncut}


%
%
%
%
%
%
%

\end{document}